\documentclass[12pt]{amsart}
\usepackage{amssymb,xcolor,graphics}
\usepackage{hyperref}
\usepackage{tabularx}
\usepackage{booktabs}
\usepackage{caption}
\usepackage{mathdots}

\newtheorem{thm}{Theorem}[section]
\newtheorem{lem}[thm]{Lemma}

\newtheorem{prop}[thm]{Proposition}
\newtheorem{ex}[thm]{Example}

\newtheorem{qu}[thm]{Question}

\newtheorem{con}[thm]{Conjecture}
\newtheorem*{prob*}{Open problem}

\theoremstyle{definition}

\newtheorem{defi}[thm]{Definition}

\theoremstyle{remark}

\newtheorem{rem}[thm]{Remark}
\newtheorem*{rem*}{Remark}

\DeclareMathOperator{\id}{id}

\DeclareMathOperator{\rad}{rad}

\newcommand{\kringel}{\mathbin{\raise1pt\hbox{$\scriptstyle\circ$}}}
\newcommand{\pkt}{\mathbin{\raise0pt\hbox{$\scriptstyle\bullet$}}}

\newcommand{\C}{\mathbb{C}}

\newcommand{\ad}{{\rm ad}}
\newcommand{\Ad}{\mathop{\rm Ad}}

\newcommand{\End}{{\rm End}}
\newcommand{\Der}{{\rm Der}}

\newcommand{\nil}{\mathop{\rm nil}}

\newcommand{\La}{\mathfrak{a}}
\newcommand{\Lb}{\mathfrak{b}}
\newcommand{\Lc}{\mathfrak{c}}

\newcommand{\Lf}{\mathfrak{f}}
\newcommand{\Lg}{\mathfrak{g}}
\newcommand{\Lh}{\mathfrak{h}}

\newcommand{\Lk}{\mathfrak{k}}
\newcommand{\Ll}{\mathfrak{l}}
\newcommand{\Ln}{\mathfrak{n}}

\newcommand{\Lr}{\mathfrak{r}}
\newcommand{\Ls}{\mathfrak{s}}
\newcommand{\Lt}{\mathfrak{t}}

\newcommand{\la}{\lambda}
\newcommand{\om}{\omega}

\newcommand{\ra}{\rightarrow}
\newcommand{\ck}{\checkmark}
\renewcommand{\phi}{\varphi}

\begin{document}


\title[Rigidity results]{Rigidity results for Lie algebras admitting a post-Lie algebra structure}
\author[D. Burde]{Dietrich Burde}
\author[K. Dekimpe]{Karel Dekimpe}
\author[M. Monadjem]{Mina Monadjem}
\address{Fakult\"at f\"ur Mathematik\\
Universit\"at Wien\\
Oskar-Morgenstern-Platz 1\\
1090 Wien \\
Austria}
\email{dietrich.burde@univie.ac.at}
\email{mina.monadjem@univie.ac.at}
\address{Katholieke Universiteit Leuven Kulak\\
E. Sabbelaan 53 bus 7657\\
8500 Kortrijk\\
Belgium}
\email{karel.dekimpe@kuleuven.be}
\date{\today}


\begin{abstract}
We study rigidity questions for pairs of Lie algebras $(\Lg,\Ln)$ admitting a post-Lie algebra structure.
We show that  if $\Lg$ is semisimple and $\Ln$ is arbitrary, then we have rigidity in the sense
that $\Lg$ and $\Ln$ must be isomorphic. The proof uses a result on the decomposition of a Lie algebra 
$\Lg=\Ls_1\dotplus \Ls_2$ as the direct vector space sum of two semisimple subalgebras. We show that 
$\Lg$ must be semisimple and hence isomorphic to the direct Lie algebra sum $\Lg\cong \Ls_1\oplus \Ls_2$. 
This solves some open existence questions for post-Lie algebra structures on pairs of Lie algebras
$(\Lg,\Ln)$. We prove additional existence results for pairs $(\Lg,\Ln)$, where $\Lg$ is complete, and for
pairs, where $\Lg$ is reductive with $1$-dimensional center and $\Ln$ is solvable or nilpotent.
\end{abstract}

\maketitle

\section{Introduction}

Post-Lie algebra structures on pairs of Lie algebras naturally arise in several areas of mathematics and physics.
Some important examples of such areas are geometric structures on manifolds, affine actions on Lie groups, Rota-Baxter operators,
\'etale and prehomogeneous modules for Lie algebras, decompositions of Lie algebras, crystallographic groups,
operad theory, deformation theory, or quantum field theory. There is a large literature on
post-Lie algebra structures, see for example the papers \cite{VAL,BU41,BU44,BU51,BU59,BU64,BU65} and the 
references given therein. \\[0.2cm]
In the study of post-Lie algebra structures one often has to investigate Lie algebra decompositions, i.e., writing a
Lie algebra $\Lg$ as the vector space sum $\Lg=\La+\Lb$ of two subalgebras $\La$ and $\Lb$. How much does the structure of
$\La$ and $\Lb$ determine the structure of $\Lg$? Recently we have studied semisimple decompositions of Lie algebras, where both
$\La$ and $\Lb$ are semisimple, see \cite{BU74}. In general, such a Lie algebra need not be semisimple. Semisimple decompositions
are closely related to prehomogeneous modules for semisimple Lie algebras. \\[0.2cm]
In the present article we show that a Lie algebra $\Lg=\Ls_1\dotplus \Ls_2$, which is the {\em direct} vector space sum of
two semisimple subalgebras $\Ls_1$ and $\Ls_2$, is already semisimple and a direct Lie algebra sum $\Lg=\Lt_1\oplus \Lt_2$ with
$\Ls_i\cong \Lt_i$ for $i=1,2$. Here we use several results about decompositions of
Lie groups and Lie algebras from the papers \cite{ON62, ON69} of Onishchik. We prove that Onishchik's arguments also imply
the following result. Let $\Ls$ be a semisimple Lie algebra which is the sum of two semisimple subalgebras, e.g., $\Ls=\Ls_1+\Ls_2$. Then
the subalgebra $\Ls_1\cap \Ls_2$ is zero or semisimple - see Lemma $\ref{2.5}$.
This is not explicitly stated in \cite{ON62, ON69}, and we could not
find it in the literature. The result is also very useful to prove a strong rigidity result for post-Lie algebra structures
on pairs $(\Lg,\Ln)$, where $\Lg$ is semisimple - see Theorem $\ref{3.3}$. \\[0.2cm]
In section $4$ we complete our ``existence table'' for post-Lie algebra structures on pairs $(\Lg,\Ln)$ from \cite{BU65}.
We only leave one case open, namely where $\Lg$ is reductive and $\Ln$ is semisimple  - see Conjecture $\ref{3.5}$.
We use the correspondence to Rota-Baxter operators to construct post-Lie algebra structures on $(\Lg,\Ln)$, where
$\Lg$ is complete. \\[0.2cm]
In section $5$ we show that there are no post-Lie algebra structures on pairs $(\Lg,\Ln)$, where $\Lg$ is reductive with
$1$-dimensional center and $\Ln$ is solvable, non-nilpotent. However, if $\Ln$ is nilpotent then we cannot show this in
general. Even the case $\Lg=\mathfrak{gl}_n(\C)$ then is open in general. Here we can at least settle the case, where $\Ln$
is $2$-step nilpotent by reducing the question to left-symmetric structures on $\Lg=\mathfrak{gl}_n(\C)$ and using decompositions
of wedge products of simple $\mathfrak{sl}_n(\C)$-modules, which we had studied already in \cite{BU74}, section $3$.
The result is that there are no post-Lie algebra structures on $(\mathfrak{gl}_n(\C),\Ln)$ for all $n\ge 2$, where $\Ln$ is
$2$-step nilpotent and non-abelian - see Proposition $5.5$.

\section{Semisimple decompositions of Lie algebras}

We first recall the definition of a post-Lie algebra structure on a pair of Lie algebras $(\Lg,\Ln)$ over a field
$K$, see \cite{BU41}:

\begin{defi}\label{pls}
Let $\Lg=(V, [\, ,])$ and $\Ln=(V, \{\, ,\})$ be two Lie brackets on a vector space $V$ over
$K$. A {\it post-Lie algebra structure}, or {\em PA-structure} on the pair $(\Lg,\Ln)$ is a
$K$-bilinear product $x\cdot y$ satisfying the identities:
\begin{align}
x\cdot y -y\cdot x & = [x,y]-\{x,y\} \label{post1}\\
[x,y]\cdot z & = x\cdot (y\cdot z) -y\cdot (x\cdot z) \label{post2}\\
x\cdot \{y,z\} & = \{x\cdot y,z\}+\{y,x\cdot z\} \label{post3}
\end{align}
for all $x,y,z \in V$.
\end{defi}

Define by  $L(x)(y)=x\cdot y$ and $R(x)(y)=y\cdot x$ the left respectively right multiplication
operators of the algebra $A=(V,\cdot)$. By \eqref{post3}, all $L(x)$ are derivations of the Lie
algebra $(V,\{,\})$. Moreover, by \eqref{post2}, the left multiplication
\[
L\colon \Lg\ra \Der(\Ln)\subseteq \End (V),\; x\mapsto L(x)
\]
is a linear representation of $\Lg$. The right multiplication $R\colon V\ra V,\; x\mapsto R(x)$
is a linear map, but in general not a Lie algebra representation. \\
If $\Ln$ is abelian, then a post-Lie algebra structure on $(\Lg,\Ln)$ corresponds to a {\it pre-Lie algebra structure},
or left-symmetric structure on $\Lg$. In other words, if $\{x,y\}=0$ for all $x,y\in V$, then the conditions reduce to
\begin{align}
x\cdot y-y\cdot x & = [x,y] \label{pre1} \\
[x,y]\cdot z & = x\cdot (y\cdot z)-y\cdot (x\cdot z), \label{pre2}
\end{align}
i.e., $x\cdot y$ is a {\it pre-Lie algebra structure} on the Lie algebra $\Lg$. \\[0.2cm]
We will assume from now on that all Lie algebras are finite-dimensional and defined over the complex numbers. \\
Let $\Lg$ be a Lie algebra and $\Ls_1$ and $\Ls_2$ be two semisimple subalgebras
of $\Lg$. We write $\Lg=\Ls_1+\Ls_2$, if $\Lg$ is the vector space sum of $\Ls_1$ and $\Ls_2$, and
\[
\Lg=\Ls_1\dotplus \Ls_2,
\]
if $\Lg$ is the {\em direct} vector space sum of $\Ls_1$ and $\Ls_2$, i.e., satisfying
$\Ls_1\cap \Ls_2=0$. We use the sum with the dot to distinguish it from the direct sum of Lie ideals $\Lg=\Ls_1\oplus \Ls_2$.
In analogy to dinilpotent groups \cite{CST}, we have introduced in \cite{BU74}
the notion of a disemisimple Lie algebra.

\begin{defi}\label{2.1}
A Lie algebra $\Lg$ is called {\em disemisimple}, if it can be written as a vector space sum of two semisimple subalgebras
$\Ls_1$ and $\Ls_2$ of $\Lg$. In this case we write $\Lg=\Ls_1+\Ls_2$. If the vector space sum is direct, we say that $\Lg$ is
{\em strongly disemisimple}.
\end{defi}

For a Lie algebra $\Lg$ denote by $\rad(\Lg)$ the solvable radical of $\Lg$ and by $\nil(\Lg)$ the nilradical of $\Lg$.
In \cite{BU74} we have studied the structure of disemisimple Lie algebras. An important tool is the close relationship
to prehomogeneous $\Ls$-modules, where $\Ls$ is a Levi subalgebra. We have shown in \cite{BU74}, Theorem $3.7$, that the solvable
radical of a disemisimple Lie algebra with a simple Levi subalgebra is {\em abelian}.  
An elementary result was Lemma $2.3$ in \cite{BU74}, which is as follows.

\begin{lem}\label{2.2}
Let $\Lg$ be a disemisimple Lie algebra. Then $\Lg$ is perfect and the solvable radical of $\Lg$
coincides with the nilradical, i.e., $\rad(\Lg)=\nil(\Lg)$. 
\end{lem}

For strongly disemisimple Lie algebras one can show more.

\begin{lem}\label{2.3}
Let $\Lg$ be a strongly disemisimple Lie algebra with $\Lg=\Ls_1\dotplus \Ls_2$. 
Then neither $\Ls_1$ nor $\Ls_2$ can be a Levi subalgebra of $\Lg$.
\end{lem}

\begin{proof}
Assume that $\Ls_1$ is a Levi subalgebra of $\Lg$ with $\Lg=\Ls_1\ltimes \rad(\Lg)$. We have
\[
\dim (\rad(\Lg))=\dim(\Lg)-\dim(\Ls_1)=\dim(\Ls_2).
\]
By Levi-Malcev there exists an element $z\in \rad(\Lg)$ such that $e^{\ad(z)}$ maps $\Ls_2$ to a subalgebra of $\Ls_1$.
Then $M=\rad(\Lg)$ is an $\Ls_2$-module with $\dim(M)=\dim(\Ls_2)$. Thus
we can apply Lemma $4.1$ of \cite{BU44}. It gives a nonzero $x\in \Ls_2$ such that $x.z=0$. So the Lie bracket
in $\Lg$ is $[x,z]=0$, so that $e^{\ad(z)}(x)=x$. This implies that also $x\in \Ls_1$, and hence $x\in \Ls_1\cap \Ls_2=0$.
We obtain $x=0$, which is a contradiction. Hence $\Ls_1$ cannot be a Levi subalgebra. The same applies to $\Ls_2$.
\end{proof}  

\begin{lem}\label{2.4}
Let $\Lg$ be a strongly disemisimple Lie algebra with $\Lg=\Ls_1\dotplus \Ls_2$, and $\Ls$
a Levi subalgebra of $\Lg$ with $\Ls_1\subseteq \Ls$. Then there exists a semisimple subalgebra $\Ls_3$ of $\Lg$
such that $\Ls=\Ls_1+\Ls_3$ and $\dim (\Ls_1\cap \Ls_3)=\dim (\rad(\Lg))$.
\end{lem}  

\begin{proof}
By Levi-Malcev there exists an element $z\in \rad(\Lg)$ such that, with $\phi=e^{\ad(z)}$, we have
$\Ls_2=\phi(\Ls_3)$ for some semisimple subalgebra $\Ls_3\subseteq \Ls$. We have $\Ls_3=\phi^{-1}(\Ls_2)$ with
$\phi^{-1}=e^{\ad(-z)}$. Let $x\in \Lg$. Then there exist $s_1\in \Ls_1$ and $s_2\in \Ls_2$ such that $x=s_1+s_2$.
For all $y\in \Lg$ we have
\[
y-\phi^{-1}(y)=y-y-\sum_{k=1}^{\infty} \frac{\ad(-z)^k}{k!}(y) \in \rad(\Lg).
\]
Hence we have, for $y=s_2$,
\[
x=s_1+s_2=s_1+\phi^{-1}(s_2)+(s_2-\phi^{-1}(s_2))\in \Ls_1+\Ls_3+\rad(\Lg).
\]
So we have
\[
\Lg\subseteq \Ls_1+\Ls_3+\rad(\Lg)\subseteq \Lg,
\]
and hence equality. Now let $x\in \Ls$. Then $x=s_1+s_3+r$ for some $s_1\in \Ls_1$, $s_3\in \Ls_3$ and $r\in \rad(\Lg)$.
It follows that $r=x-s_1-s_3\in \Ls$, so that $r=0$ because of $\rad(\Lg)\cap \Ls=0$.
Hence we have $\Ls\subseteq \Ls_1+\Ls_3\subseteq \Ls$, and $\Ls=\Ls_1+\Ls_3$. For the dimensions we have
\[
\dim(\Ls)=\dim(\Ls_1)+\dim(\Ls_3)-\dim (\Ls_1\cap \Ls_3).
\]
On the other hand we have $\Ls_1\dotplus \Ls_2=\Ls\ltimes \rad(\Lg)$ and $\dim(\Ls_2)=\dim(\Ls_3)$, so that
\[
\dim (\Ls)=\dim (\Ls_1)+\dim (\Ls_3)-\dim (\rad(\Lg)).
\]
So we have $\dim (\Ls_1\cap \Ls_3)=\dim (\rad(\Lg))$.
\end{proof}

The following lemma is a consequence of results and arguments from the two papers \cite{ON62} and \cite{ON69} by Onishchick.

\begin{lem}\label{2.5}
Let $\Ls$ be a semisimple and disemisimple Lie algebra with $\Ls=\Ls_1+\Ls_2$. Then the 
subalgebra $\Ls_1\cap \Ls_2$ is zero or semisimple.
\end{lem}  

\begin{proof}
Let $G_{\Ls}$ be the connected algebraic group corresponding to $\Ls$ and let $G_{\Ls_1}$ and 
$G_{\Ls_2}$ be the connected algebraic subgroups of $G_{\Ls}$ corresponding to the subalgebras 
$\Ls_1$ and $\Ls_2$. In \cite[pages 522-523]{ON69} in the proof of Theorem 3.1 in the case of 
reductive algebraic decompositions over $\C$ it is shown that $G_{\Ls} = G_{\Ls_1} G_{\Ls_2}$ 
(and so also $G_{\Ls} = G_{\Ls_2} G_{\Ls_1}$).\\
From \cite[Corollary 3.1]{ON69} we then get that there exists a real compact form $\Lk$ of $\Ls$ 
and real compact forms $\Lk_1$ and $\Lk_2$ of the Lie algebras $\Ad(x)(\Ls_1)$ and $\Ad(y)(\Ls_2)$ 
corresponding to some conjugates $xG_{\Ls_1} x^{-1}$ and $yG_{\Ls_2}y^{-1}$ of $G_{\Ls_1}$ and 
$G_{\Ls_2}$ such that $\Lk= \Lk_1 + \Lk_2$. In \cite{ON62}, page $18$, after Corollary $2$ it is 
shown that for the decomposition $\Lk=\Lk_1+\Lk_2$ of compact forms we have that $\Lk_1\cap \Lk_2$ 
is zero or {\em semisimple} (Onishchik writes $U_0=G'_0\cap G''_0$ for this intersection). 
From this it follows that the complexification $( \Lk_1\cap \Lk_2) ^\C = \Lk_1^\C\cap \Lk_2^\C
= \Ad(x)(\Ls_1) \cap \Ad(y)(\Ls_2)$ is also zero or semisimple. \\
We now claim that $\Ls_1 \cap \Ls_2$ is isomorphic to $\Ad(x)(\Ls_1) \cap \Ad(y)(\Ls_2)$. 
To see this, first write $x=ba$ with $b \in G_{\Ls_2}$ and $a \in G_{\Ls_1}$ 
(recall that $G_\Ls= G_{\Ls_2}G_{\Ls_1}$). Then 
\begin{eqnarray*}
\Ad(x)(\Ls_1) \cap \Ad(y)\Ls_2 & =  &\Ad(b) ( \Ad(a)(\Ls_1) \cap \Ad (b^{-1}y) (\Ls_2))\\
& = & \Ad(b) (\Ls_1 \cap \Ad (b^{-1}y) (\Ls_2))\\
& \cong& \Ls_1 \cap \Ad (b^{-1}y) (\Ls_2 ).
\end{eqnarray*}
Now write $b^{-1}y=a'b'$ with $a'\in G_{\Ls_1}$ and $b'\in G_{\Ls_2}$, then 
\begin{eqnarray*}
 \Ls_1 \cap \Ad (b^{-1}y) (\Ls_2 ) & = & \Ls_1 \cap \Ad(a') \Ad(b') (\Ls_2)\\
    & = & \Ls_1 \cap   \Ad(a')  (\Ls_2)\\
    & = & \Ad(a') ( \Ls_1 \cap \Ls_2) \\
    & \cong & \Ls_1 \cap \Ls_2,
\end{eqnarray*}
from which we indeed deduce that $\Ls_1 \cap \Ls_2\cong \Ad(x)(\Ls_1) \cap \Ad(y)\Ls_2 $ 
and so is semisimple.
\end{proof}  

Now we can prove the main decomposition theorem.

\begin{thm}\label{2.6}
Let $\Lg$ be a strongly disemisimple Lie algebra with $\Lg=\Ls_1\dotplus \Ls_2$. 
Then $\Lg$ is semisimple and isomorphic to the direct Lie algebra sum $\Ls_1\oplus \Ls_2$.
\end{thm}  

\begin{proof}
Assume that $\Lg$ is not semisimple. Then we have $\dim(\rad(\Lg))\ge 1$. Let $\Ls$ be a Levi subalgebra of $\Lg$
such that $\Ls_1\subseteq \Ls$. By Lemma $\ref{2.4}$ there exists a semisimple subalgebra $\Ls_3$ such that
$\Ls=\Ls_1+\Ls_3$ and $\dim (\Ls_1\cap \Ls_3)=\dim (\rad(\Lg))\ge 1$. It follows that
$\Ls_1\cap \Ls_3\subseteq \Ls$ is a semisimple subalgebra by Lemma $\ref{2.5}$. Hence $\Ls_1\cap \Ls_3$ is a 
Levi subalgebra of the Lie algebra $(\Ls_1\cap \Ls_3)\ltimes \rad(\Lg)$.
Let $\phi=e^{\ad(z)}$ the special automorphism with $\phi(\Ls_3)=\Ls_2$ from the proof of Lemma $\ref{2.4}$.
We claim that
\[
(\Ls_1\cap \Ls_3)\ltimes \rad(\Lg)=(\Ls_1\cap \Ls_3)\dotplus \phi (\Ls_1\cap \Ls_3).
\]
Hence $\Ls_1\cap \Ls_3$ cannot be a Levi subalgebra of $(\Ls_1\cap \Ls_3)\ltimes \rad(\Lg)$ by 
Lemma $\ref{2.3}$.
This is a contradiction and it follows that $\Lg$ is semisimple. We need to show the claimed equality. 
Let $y\in \rad(\Lg)$. Then
\[
y=s_1+s_2=s_1+\phi(s_3)=(s_1+s_3)+(\phi(s_3)-s_3)
\]  
for some $s_i\in \Ls_i$. Here $\phi(s_3)-s_3\in \rad(\Lg)$ and $s_1+s_3\in \Ls$, so that $s_1+s_3=0$ and $y=\phi(s_3)-s_3$.
Note that $s_3=-s_1\in \Ls_1\cap \Ls_3$. Now let $x\in \Ls_1\cap \Ls_3$ and $y\in \rad(\Lg)$. Then 
\[
x+y=x+\phi(s_3)-s_3=(x-s_3)+\phi(s_3) \in (\Ls_1\cap \Ls_3)\dotplus \phi (\Ls_1\cap \Ls_3).
\]
Conversely let $y,z\in \Ls_1\cap \Ls_3$. Then
\[
y+\phi(z)=(y+z)+(\phi-\id)(z) \in (\Ls_1\cap \Ls_3)\ltimes \rad(\Lg).
\]
Finally, since $\Ls_1$ and $\Ls_2$ are semisimple subalgebras, they are {\em reductive in} the 
semisimple Lie algebra $\Lg$. Hence we can apply Koszul's Theorem \cite{KOS} to conclude that 
$\Lg$ is isomorphic to the direct Lie algebra ideal sum of $\Ls_1$ and $\Ls_2$.
\end{proof}  

\begin{rem}
The theorem cannot be generalized to disemisimple Lie algebras. Indeed, we have shown in
Example $4.10$ of \cite{BU64}, that for $n\ge 2$,
\[
\Ls\Ll_n(\C)\ltimes V(n)=\Ls\Ll_n(\C)+\phi(\Ls\Ll_n(\C))=\Ls_1+\Ls_2
\]
is the vector space sum of two simple subalgebras $\Ls_1$ and $\Ls_2$. The Lie algebra 
$\Ls\Ll_n(\C)\ltimes V(n)$ is perfect, but not semisimple. We have $\phi=e^{\ad(z)}$ for a certain  $z$ in the radical $V(n)$, which is abelian. Note that the intersection satisfies
\[
\dim(\Ls_1\cap \Ls_2)=n^2-n-1.
\]  
\end{rem}

\section{Rigidity results}

Let $(\Lg,\Ln)$ be a pair of finite-dimensional complex Lie algebras admitting a post-Lie algebra structure.
The rigidity question is the following. \\[0.2cm]
{\em For which algebraic properties of $\Lg$ and $\Ln$ can we conclude
  that $\Lg$ and $\Ln$ are necessarily isomorphic?} \\[0.2cm]
The properties we are mostly interested in here are that $\Lg$ and $\Ln$ are {\em simple, semisimple or reductive}. \\
If both Lie algebras are solvable, it is obvious that $\Lg$ and $\Ln$ need not be isomorphic in general.
Indeed, consider the Lie algebra $\Lg=\Lr_{3,1}(\C)$ with basis $(e_1,e_2,e_3)$ and Lie brackets $[e_1,e_2]=e_2$, $[e_1,e_3]=e_3$,
and let $\Ln=\Lr_3(\C)$, given by the Lie brackets $\{ e_1,e_2\}=e_2$, $\{e_1,e_3\}=e_2+e_3$. The Lie algebras are
solvable, but not isomorphic. Nevertheless there exists a post-Lie algebra structure on $(\Lg,\Ln)$:

\begin{ex}
There exists a post-Lie algebra structure on the pair $(\Lg,\Ln)=(\Lr_{3,1}(\C),\Lr_3(\C))$, given
by $e_1\cdot e_3=-e_2$ and all other products $e_i\cdot e_j$ equal to zero.
\end{ex}

If $\Lg$ and $\Ln$ are both simple, then we know that rigidity holds. This is Proposition 
$4.6$ in \cite{BU41}.
However, an even stronger result holds if just $\Lg$ is simple, see Theorem $3.1$ of \cite{BU51}.

\begin{prop}
Let $(\Lg,\Ln)$ be a pair of Lie algebras, where $\Lg$ is simple and $\Ln$ is arbitrary.
Suppose that $(\Lg,\Ln)$ admits a post-Lie algebra structure. Then $\Ln$ is isomorphic to $\Lg$.
\end{prop}  

We can now generalize this result to the case where $\Lg$ is semisimple.

\begin{thm}\label{3.3}
Let $(\Lg,\Ln)$ be a pair of Lie algebras, where $\Lg$ is semisimple and $\Ln$ is arbitrary.
Suppose that $(\Lg,\Ln)$ admits a post-Lie algebra structure. Then $\Ln$ is isomorphic to $\Lg$.
\end{thm}  

\begin{proof}
By Proposition $2.11$ of \cite{BU41} there is an embedding
\[
\phi\colon \Lg\hookrightarrow \Ln\rtimes \Der(\Ln)
\]
such that $p\circ \phi\colon \Lg\ra \Ln$ is a vector space isomorphism, where $p\colon \Ln\rtimes \Der(\Ln)\ra \Ln$
denotes the projection on $\Ln$. If $q\colon \Ln\rtimes \Der(\Ln)\ra \Der(\Ln)$ denotes the projection on the second
factor, then $q\circ \phi \colon \Lg\ra \Der(\Ln)$ is a Lie algebra homomorphism. Hence
$\Lh=(q\circ \phi)(\Lg)\subseteq \Der(\Ln)$ is a semisimple Lie algebra, because $\Lg$ is semisimple.
Note that if $\Lh=0$ we obtain $\Lg\cong \Ln$ and we are done. We may view the embedding now as
$\phi\colon \Lg\hookrightarrow \Ln\rtimes \Lh$.
We claim that
\[
\Ln\rtimes \Lh = \phi(\Lg)\dotplus \Lh
\]
as a direct vector space sum of subalgebras of $\Ln\rtimes \Lh$. Indeed, for a given element $(x,y)\in \Ln\rtimes \Lh$ there
is a unique $v\in \Lg$ with $\phi(v)=(x,z)$ for some $z\in \Lh$, and there is a unique $w\in \Lh$ with $(x,y)=\phi(v)+w$.
Hence $\Ln\rtimes \Lh$ is the direct vector space sum of two semisimple subalgebras $\phi(\Lg)$ and $\Lh$.
By Theorem $\ref{2.6}$ it follows that $\Ln\rtimes \Lh$ is semisimple, and hence also the ideal $\Ln$ is semisimple.
Hence we have $\Ln\rtimes \Lh \cong \Ln\oplus \Lh$, and also $\Ln\rtimes \Lh\cong \phi(\Lg)\oplus \Lh$. Writing
\begin{align*}
\phi(\Lg) & \cong \La_1\oplus \cdots \oplus \La_i \\
\Ln & \cong  \Lb_1\oplus \cdots \oplus \Lb_j \\
\Lh & \cong \Lc_1\oplus \cdots \oplus \Lc_k
\end{align*}
as direct sum of simple ideals, $\phi(\Lg)\oplus \Lh\cong \Ln\oplus \Lh$ implies that
\[
(\La_1\oplus \cdots \oplus \La_i) \oplus  (\Lc_1\oplus \cdots \oplus \Lc_k) \cong (\Lb_1\oplus \cdots \oplus \Lb_j)\oplus
(\Lc_1\oplus \cdots \oplus \Lc_k). 
\]
Since the decomposition of a semisimple Lie algebra into simple ideals is unique up to permutation, we obtain
$i=j$ and $\Lg\cong \phi(\Lg)\cong \Ln$.
\end{proof}

\begin{rem}
A corresponding strong rigidity result for the case that $\Ln$ is semisimple does not hold. For example, we know
that there are post-Lie algebra structures on pairs $(\Lg,\Ln)$, where $\Ln$ is semisimple and $\Lg$ is solvable,
see Proposition $3.1$ in \cite{BU44}. Furthermore, even if $(\Lg,\Ln)$ is a pair of reductive Lie algebras admitting a post-Lie algebra
structure, $\Lg$ need not be isomorphic to $\Ln$ in general. Indeed, for $\Lg= \Ls\Ll_n(\C)\oplus \Lg\Ll_n(\C)$ and
$\Ln= \Ls\Ll_n(\C)\oplus \C^{n^2}$ the pair $(\Lg,\Ln)$ admits a post-Lie algebra structure, by taking the direct sum of a non-trivial
post-Lie algebra structure on $(\Lg\Ll_n(\C),\C^{n^2})$, coming from a pre-Lie algebra structure on $\Lg\Ll_n(\C)$, and the zero structure 
on $(\Ls\Ll_n(\C),\Ls\Ll_n(\C))$.
\end{rem}  

On the other hand, we still believe that there holds the following rigidity result. 

\begin{con}\label{3.5}
Let $(\Lg,\Ln)$ be a pair of Lie algebras, where $\Lg$ is reductive and $\Ln$ is semisimple.
Suppose that $(\Lg,\Ln)$ admits a post-Lie algebra structure. Then $\Ln$ is isomorphic to $\Lg$.
\end{con}

\section{The existence table}

For the existence question of post-Lie algebra structures on pairs of Lie algebras $(\Lg,\Ln)$ we have introduced
a ``table'' in \cite{BU65}, where $\Lg$ and $\Ln$ belong to one of the following seven classes, defined by the following
properties: {\em abelian, nilpotent, solvable, simple, semisimple, reductive, complete}. Here we want to avoid an unnecessary
overlap, so we assume that nilpotent means non-abelian, solvable means non-nilpotent,
semisimple means non-simple, and reductive respectively complete means non-semisimple. Of course,  
a complete Lie algebra in the table may also be solvable and non-nilpotent. It cannot be nilpotent, since nilpotent Lie algebras
have a non-trivial center. Similarly, a reductive, non-semisimple Lie algebra may be abelian, but
has a non-trivial center and hence cannot be complete. The existence table in \cite{BU65} still has six open cases:
\vspace*{0.5cm}
\begin{center}
\begin{tabular}{l|lllllll}
$(\Lg,\Ln)$ & \color{blue}{$\Ln$ abe} & \color{green}{$\Ln$ nil} & $\Ln$ sol & \color{red}{$\Ln$ sim}
& \color{orange}{$\Ln$ sem} & \color{purple}{$\Ln$ red} &  \color{cyan}{$\Ln$ com} \\[1pt]
\hline
\color{blue}{$\Lg$ abelian} & $\ck$ & $\ck$ & $\ck$ & $-$ & $-$ & $-$ & $\ck$ \\[1pt]
\color{green}{$\Lg$ nilpotent} & $\ck$ & $\ck$ & $\ck$ & $-$ & $-$ & $-$ & $\ck$ \\[1pt]
 $\Lg$ solvable & $\ck$ & $\ck$ & $\ck$ & $\ck$ & $\ck$ & $\ck$ & $\ck$ \\[1pt]
\color{red}{$\Lg$ simple} & $-$ & $-$ & $-$ & $\ck$ & $-$ & $-$ & $-$ \\[1pt]
\color{orange}{$\Lg$ semisimple} & $-$ & $-$ & $-$ & $-$ & $\ck$ & $(1)$ & $-$ \\[1pt]
\color{purple}{$\Lg$ reductive}  & $\ck$ & $(3)$ & $(4)$ & $-$ & $(2)$ & $\ck$ & $\ck$ \\[1pt]
\color{cyan}{$\Lg$ complete}  & $\ck$ & $\ck$ & $\ck$ & $(5)$ & $(6)$ & $\ck$ & $\ck$ \\[1pt]
\end{tabular}
\end{center}
\vspace*{0.5cm}
We can now solve all open cases with one exception and obtain the following table.

\begin{thm}\label{4.1}
The existence table for post-Lie algebra structures on pairs $(\Lg,\Ln)$ is given as follows:
\vspace*{0.5cm}
\begin{center}
\begin{tabular}{l|lllllll}
$(\Lg,\Ln)$ & \color{blue}{$\Ln$ abe} & \color{green}{$\Ln$ nil} & $\Ln$ sol & \color{red}{$\Ln$ sim}
& \color{orange}{$\Ln$ sem} & \color{purple}{$\Ln$ red} &  \color{cyan}{$\Ln$ com} \\[1pt]
\hline
\color{blue}{$\Lg$ abelian} & $\ck$ & $\ck$ & $\ck$ & $-$ & $-$ & $-$ & $\ck$ \\[1pt]
\color{green}{$\Lg$ nilpotent} & $\ck$ & $\ck$ & $\ck$ & $-$ & $-$ & $-$ & $\ck$ \\[1pt]
 $\Lg$ solvable & $\ck$ & $\ck$ & $\ck$ & $\ck$ & $\ck$ & $\ck$ & $\ck$ \\[1pt]
\color{red}{$\Lg$ simple} & $-$ & $-$ & $-$ & $\ck$ & $-$ & $-$ & $-$ \\[1pt]
\color{orange}{$\Lg$ semisimple} & $-$ & $-$ & $-$ & $-$ & $\ck$ & $-$ & $-$ \\[1pt]
\color{purple}{$\Lg$ reductive}  & $\ck$ & $\ck$ & $\ck$ & $-$ & \color{red}{$?$} & $\ck$ & $\ck$ \\[1pt]
\color{cyan}{$\Lg$ complete}  & $\ck$ & $\ck$ & $\ck$ & $\ck$ & $\ck$ & $\ck$ & $\ck$ \\[1pt]
\end{tabular}
\end{center}
\vspace*{0.5cm}
A checkmark only means that there is \texttt{some} non-trivial pair $(\Lg,\Ln)$ of Lie algebras with the
given algebraic properties admitting a post-Lie algebra structure. A dash means that there does not exist any
post-Lie algebra structure on such a pair.
\end{thm}

\begin{proof}
By Theorem $\ref{3.3}$ there is no post-Lie algebra structure on a pair $(\Lg,\Ln)$, where $\Lg$ is semisimple,
and $\Ln$ is not semisimple. This solves case $(1)$. 
Case $(2)$ is still open, see Conjecture $\ref{3.5}$.
For the remaining cases, there exist non-trivial examples admitting a post-Lie algebra structure. The cases
$(3),(4)$ follow from Proposition $\ref{4.2}$, the case $(5)$ from Proposition $\ref{4.3}$ and the case $(6)$
from Proposition $\ref{4.4}$ below.
\end{proof}

Denote by $\Ln_3(\C)$ the Heisenberg Lie algebra and by $\Lr_2(\C)$ the $2$-dimensional non-abelian
Lie algebra. 

\begin{prop}\label{4.2}
There exists a post-Lie algebra structure on the pairs
\begin{align*}
(\Lg,\Ln) & =(\Ls\Ll_2(\C)\oplus \C^4,\C^4\oplus \Ln_3(\C)),\\
(\Lg,\Ln) & =(\Ls\Ll_2(\C)\oplus \C^3,\C^4\oplus \Lr_2(\C)).
\end{align*}            
\end{prop}

\begin{proof}
Let $x\cdot y$ be a LSA, or pre-Lie algebra structure on $\Ls\Ll_2(\C)\oplus \C$. Such structures exist and have been
classified in \cite{BU4}, Theorem $3$. By definition this structure is a post-Lie algebra structure on the pair
$(\Ls\Ll_2(\C)\oplus \C,\C^4)$. Let $x\circ y$ be an LR-structure on $\Ln_3(\C)$. Such structures exist and have been classified
in \cite{BU34}, Proposition $3.1$. By definition this structure is a post-Lie algebra structure on the pair $(\C^3,\Ln_3(\C))$.
Now equip the pair $(\Lg,\Ln)=(\Ls\Ll_2(\C)\oplus \C^4,\C^4\oplus \Ln_3(\C))$ with the post-Lie algebra structure
\[
(x_1,y_1)\bullet (x_2,y_2)=(x_1\cdot x_2, y_1\circ y_2)
\]
induced by the post-Lie algebra structures on $(\Ls\Ll_2(\C)\oplus \C,\C^4)$ and $(\C^3,\Ln_3(\C))$.
This defines a post-Lie algebra structure on $(\Ls\Ll_2(\C)\oplus \C^4,\C^4\oplus \Ln_3(\C))$. \\[0.2cm]
The same construction works by choosing an LR-structure $x\circ y$ on $\Lr_2(\C)$, which is a post-Lie algebra structure on the pair
$(\C^2,\Lr_2(\C)$. The post-Lie algebra structures on $(\Ls\Ll_2(\C)\oplus \C,\C^4)$ and $(\C^2,\Lr_2(\C)$ induce a post-Lie algebra
structure on the direct sum of these pairs, e.g., on $(\Ls\Ll_2(\C)\oplus \C^3,\C^4\oplus \Lr_2(\C))$.
\end{proof}

Denote by $E_{ij}$ the matrix with entry $1$ at position $(i,j)$, and all other entries equal to $0$. Let $\Ln=\Ls\Ll_3(\C)$
and consider the following basis for it:
\begin{align*}
e_1 & =E_{12},\; e_2 = E_{13},\; e_3=E_{21},\;  e_4=E_{23},\;e_5 =E_{31}, \\
e_6 & = E_{32},\; e_7=E_{11}-E_{22},\;  e_8=E_{22}-E_{33}.
\end{align*}
Then the Lie brackets are defined by
\begin{align*}
\{e_1,e_3 \} & = e_7,     & \{e_2,e_6 \} & = e_1,   &  \{e_4,e_6 \} & = e_8,   \\
\{e_1,e_4 \} & = e_2,     & \{e_2,e_7 \} & = -e_2,  &  \{e_4,e_7 \} & = e_4,   \\
\{e_1,e_5 \} & = -e_6,    & \{e_2,e_8 \} & = -e_2,  &  \{e_4,e_8 \} & = -2e_4,  \\
\{e_1,e_7 \} & = -2e_1,   & \{e_3,e_6 \} & = -e_5,  &  \{e_5,e_7 \} & = e_5,   \\
\{e_1,e_8 \} & = e_1,     & \{e_3,e_7 \} & = 2e_3,  &  \{e_5,e_8 \} & = e_5,  \\
\{e_2,e_3 \} & = -e_4,    & \{e_3,e_8 \} & = -e_3,  &  \{e_6,e_7 \} & = -e_6, \\
\{e_2,e_5 \} & = e_7+e_8, & \{e_4,e_5 \} & = e_3,   &  \{e_6,e_8 \} & = 2e_6.
\end{align*}

Let $\La\Lf\Lf_n(\C)=\mathfrak{gl}_n(\C)\ltimes \C^n$ be the affine Lie algebra of dimension $n^2+n$. It is
simply complete, i.e., it is complete and has no nontrivial complete ideal. We can choose a basis $(f_1,\ldots ,f_8)$ for
$\mathfrak{aff}_2(\C)\oplus \mathfrak{aff}_1(\C)$ with Lie brackets
\begin{align*}
[f_1,f_2] & = f_2,    & [f_2,f_4] & = f_2,    &  [f_4,f_6] & = f_6,\\
[f_1,f_3] & = -f_3    & [f_2,f_6] & = f_5,    &  [f_7,f_8] & = f_7,\\
[f_1,f_5] & = f_5,    & [f_3,f_4] & =-f_3,    \\
[f_2,f_3] & =f_1-f_4, & [f_3,f_5] & = f_6,
\end{align*}
where $(f_1,\ldots f_6)=(E_{11}, E_{12}, E_{21}, E_{22}, E_{13}, E_{23})$ is a basis of $\mathfrak{aff}_2(\C)$ and
$(f_7,f_8)$ is  a basis of $\mathfrak{aff}_1(\C)=\Lr_2(\C)$.

\begin{prop}\label{4.3}
Let $\Ln=\Ls\Ll_3(\C)$. Then there exists a post-Lie algebra structure on the pair $(\Lg,\Ln)$ with
$\Lg\cong \La\Lf\Lf_2(\C)\oplus \La\Lf\Lf_1(\C)$. More precisely, $x\cdot y=\{\phi(x),y\}$ with
\begin{align*}
\phi(e_3) & =-e_3-e_7, \\
\phi(e_4) & =-e_4-e_5, \\
\phi(e_i) & =0 \text{ for } i=1,2,5,6,7,8
\end{align*}  
defines an inner post-Lie algebra structure on the pair $(\Lg,\Ln)$ with Lie brackets for $\Lg$ given by
\begin{align*}
[e_1,e_3] & = 2e_1,     & [e_2,e_5] & = e_7+e_8, & [e_4,e_7]& =-e_5,\\
[e_1,e_4] & = e_6,      & [e_2,e_6] & = e_1,     & [e_4,e_8] & = -e_5,\\
[e_1,e_5] & = -e_6,     & [e_2,e_7] & = -e_2,    & [e_5,e_7] & = e_5,\\
[e_1,e_7] & = -2e_1,    & [e_2,e_8] & = -e_2,    & [e_5,e_8]& = e_5,\\
[e_1,e_8] & = e_1,      & [e_3,e_4] & =  e_4,    & [e_6,e_7] & = -e_6,\\
[e_2,e_3] & = e_2,      & [e_3,e_5] & = e_5,     & [e_6,e_8] & = 2e_6,\\
[e_2,e_4] & = -e_7-e_8, & [e_3,e_6] & =-e_6,     &           &
\end{align*}
and the post-Lie algebra structure given by
\begin{align*}
e_3\cdot e_1 & = -2e_1+e_7,   & e_3\cdot e_6 & = e_5-e_6,  & e_4\cdot e_4 & = e_3,\\
e_3\cdot e_2 & = -e_2-e_4,    & e_3\cdot e_7 & = -2e_3,    & e_4\cdot e_5 & = -e_3, \\
e_3\cdot e_3 & = 2e_3,        & e_3\cdot e_8 & = e_3,      & e_4\cdot e_6 & = -e_8, \\
e_3\cdot e_4 & = e_4,         & e_4\cdot e_1 & = e_2-e_6,  & e_4\cdot e_7 & = -e_4-e_5,\\
e_3\cdot e_5 & = e_5,         & e_4\cdot e_2 & = e_7+e_8,  & e_4\cdot e_8 & = 2e_4-e_5.
\end{align*}
\end{prop}

\begin{proof}
By Corollary $2.15$ in \cite{BU59}, post-Lie algebra structures on $(\Lg,\Ln)$ are in bijection to Rota-Baxter operators
of weight $\la=1$ on $\Ln$, since $\Ln$ is complete. For $x\cdot y=\{\phi(x),y\}$ we can choose
$\phi(e_i)=0$ for $i=1,2,5,6,7,8$. Then a short computation shows that the
above homomorphism $\phi$ is a possible solution. We have a Lie algebra isomorphism
\[
f\colon \Lg\ra \La\Lf\Lf_2(\C)\oplus \La\Lf\Lf_1(\C)
\]
given by
\[
f=\begin{pmatrix} 0 & 0  & -2 & 0  & 0  & 0 & 2  & -1 \cr
                  0 & 1  & 0  & 0  & 0  & 0 & 0  & 0 \cr
                  0 & 0  & 0  & -1 & 1  & 0 & 0  & 0 \cr
                  0 & 0  & -1 & 0  & 0  & 0 & 1  & -2 \cr
                  1 & 0  & 0  & 0  & 0  & 0 & 0  & 0 \cr
                  0 & 0  & 0  & 0  & 0  & 1 & 0  & 0 \cr
                  0 & 0  & 0  & 1  & 0  & 0 & 0  & 0 \cr
                  0 & 0  & -1 & 0  & 0  & 0 & 0  & 0 \cr
  \end{pmatrix}
\]
where $\det(f)=-3$.
\end{proof}

For the next result let
\[
  e_1=E_{12}, e_2=E_{21}, e_3=E_{11}-E_{22}, e_4=E_{45},e_5=E_{54},e_6=E_{44}-E_{55}
\]
be a basis for the Lie algebra $\Ln=\Ls\Ll_2(\C)\oplus \Ls\Ll_2(\C)$. 

\begin{prop}\label{4.4}
Let $\Ln=\Ls\Ll_2(\C)\oplus \Ls\Ll_2(\C)$. Then there exists a post-Lie algebra structure on the pair $(\Lg,\Ln)$ with
$\Lg\cong \Lr_2(\C)\oplus \Lr_2(\C)\oplus \Lr_2(\C)$. It is given by $x\cdot y=\{\phi(x),y\}$ with
\[
  \phi=\begin{pmatrix} 0 & 0  & 0  & 0  & 0  & 0 \cr
                       0 & -1 & 0  & 0  & 0  & 0 \cr
                       0 & 0  & 0  & 0  & -1 & 0 \cr
                       0 & 0  & 0  & 0  & 0  & 0 \cr
                       0 & 0  & 0  & 0  & -1 & 0 \cr
                       0 & 0  & 0  & 0  & 0  & 0 \end{pmatrix}
\]
where the Lie brackets for $\Lg$ given by
\begin{align*}
[e_1,e_3] & = -2e_1,   & [e_2,e_5] & =-2e_2,  & [e_4,e_6]& =-2e_4,\\
[e_1,e_5] & = 2e_1,    &   &
\end{align*}
and the post-Lie algebra structure is given explicitly by
\begin{align*}
e_2\cdot e_1 & = e_3,      & e_5\cdot e_1 & =-2e_1,   & e_5\cdot e_4 & =e_6,\\
e_2\cdot e_3 & = -2e_2,    & e_5\cdot e_2 & =2e_2,   & e_5\cdot e_6 & =-2e_5. \\
\end{align*}
\end{prop}

\begin{proof}
By Corollary $2.15$ in \cite{BU59}, post-Lie algebra structures on $(\Lg,\Ln)$ are in bijection to Rota-Baxter operators
of weight $\la=1$ on $\Ln$, since $\Ln$ is complete. We can write $\Ln=\Ln_1\dotplus \Ln_2$ as the direct vector space
sum of the two subalgebras
\[
\Ln_1=\langle e_1,e_3,e_4,e_6\rangle, \; \Ln_2=\langle e_2,e_3+e_5\rangle.
\]
By Proposition $2.7$ in \cite{BU59} we know that $R(n_1+n_2)=-n_2$ for all $n_1\in \Ln_1,n_2\in \Ln_2$ defines a Rota-Baxter operator
of weight $1$ on $\Ln$ with associated post-Lie algebra structure given by $x\cdot y=\{R(x),y\}$. So we have
$R(e_1)=R(e_3)=R(e_4)=R(e_6)=0$ and $R(e_2)=-e_2$, $R(e_3+e_5)=-e_3-e_5$.
This gives the above matrix $\phi$ for the operator $R$. One verifies the Lie brackets for $\Lg$ coming from
$[x,y]=x\cdot y-y\cdot x+\{x,y\}$ and the explicit post-Lie algebra structure with respect to the basis of $\Ln$. We have
the decomposition of $\Lg$ into ideals
\begin{align*}
  \Lg & = \Ln_1\oplus \Ln_2 \\
      & = \langle e_1,e_3, e_4,e_6\rangle \oplus \langle e_2,e_3+e_5\rangle \\
      & \cong  \Lr_2(\C)\oplus \Lr_2(\C)\oplus \Lr_2(\C).
\end{align*}
\end{proof}

\section{Reductive Lie algebras with one-dimensional center}

In this section we study the existence of post-Lie algebra structures on pairs $(\Lg,\Ln)$, where $\Lg$
is reductive with $1$-dimensional center with $\dim (\Lg)\ge 2$. It is well known that
there are no post-Lie algebra structures on pairs $(\Lg,\Ln)$ where $\Lg$ is semisimple
and $\Ln$ is abelian. Hence it is natural to consider the case, where $\Lg$ is reductive 
with a $1$-dimensional center and $\Ln$ is solvable, nilpotent or abelian. For an abelian Lie algebra 
$\Ln$, post-Lie algebra structures on $(\Lg,\Ln)$ are pre-Lie algebra structures on $\Lg$. For the case 
$\Lg=\mathfrak{gl}_n(\C)$ there exist pre-Lie algebra structures, and one can even classify them, 
see \cite{BU4}. So the question is, what we can say when $\Ln$ is solvable or nilpotent.
We have the following result.

\begin{thm}\label{5.1}
Let $(\Lg,\Ln)$ be a pair of Lie algebras, where $\Lg$ is reductive with $1$-dimensional 
center, and $\Ln$ is solvable non-nilpotent. Then there is no post-Lie algebra structure on $(\Lg,\Ln)$.
\end{thm}

\begin{proof}
Assume that $x\cdot y$ is a post-Lie algebra structure on $(\Lg,\Ln)$. Denote by $V$ the underlying 
vector space of $\Lg$ and $\Ln$ with $\dim(V)=n$. All left multiplications $L(x)$ are derivations 
of $\Ln$ and hence map $\Ln$ into the nilradical $\nil(\Ln)$. Indeed, for any Lie algebra $\Lh$ we have 
$D(\rad(\Lh))\subseteq \nil(\Lh))$ for all $D\in \Der(\Lh)$, and for solvable $\Lh$ we clearly have
$\Lh=\rad(\Lh)$.
So we have
\[
\Ln\cdot \Ln=L(\Ln)(\Ln) \subseteq \nil(\Ln).
\]
Since $\Ln$ is solvable, $\{\Ln,\Ln\}$ is a nilpotent ideal in $\Ln$, so that
$\{\Ln,\Ln\}\subseteq \nil(\Ln)$. Hence for $x,y\in V$ the Lie bracket of $[x,y]$ in $\Lg$ satisfies
\[
[x,y]=x\cdot y-y\cdot x+\{x,y\} \in \nil (\Ln).
\]
So we have $[\Lg,\Lg]\subseteq \nil(\Ln)$. However, we have $\dim ([\Lg,\Lg])=n-1$, because $\Lg=[\Lg,\Lg]\oplus Z(\Lg)$.
It follows that $\dim \nil(\Ln)\ge n-1$.
Since $\Ln$ is not nilpotent, $\dim \nil(\Ln)\le n-1$. Together we obtain that $\dim \nil(\Ln)=n-1$ and $[\Lg,\Lg]=\nil(\Ln)$
as vector spaces. But this implies that the post-Lie algebra structure on $(\Lg,\Ln)$ restricts to a post-Lie algebra
structure on the pair $([\Lg,\Lg],\nil(\Ln))$, which is impossible by Theorem $4.2$ of \cite{BU44}, because $[\Lg,\Lg]$
is semisimple.
\end{proof}

In particular, there are no post-Lie algebra structures on $(\mathfrak{gl}_n(\C),\Ln)$ for solvable non-nilpotent Lie algebras
$\Ln$. However, this is no longer true if we replace $\Lg\Ll_n(\C)$ by $\Lg\Ll_n(\C)\oplus \C$, i.e., if we consider a
reductive Lie algebra with $\dim (Z(\Lg))\ge 2$. Indeed, we have the following result.

\begin{prop}\label{5.2}
For any $n\ge 2$ there is a post-Lie algebra structure on the pair $(\Lg,\Ln)$, where
$\Lg=\Lg\Ll_n(\C)\oplus \C$ with basis $(y_1,\ldots ,y_{n^2},x)$, where $x$ spans $\C$, and let
$\Ln$ be the $2$-step solvable Lie algebra with Lie brackets $\{x,y_i\}=y_i$ for $1\le i\le n^2$. 
\end{prop}

\begin{proof}
By definition, $\La=\langle y_1,y_2,\ldots ,y_{n^2}\rangle$ is an abelian ideal in $\Ln$ with $\Ln=\La\rtimes \langle x\rangle$,
where $x$ acts on $\La$ by the derivation $\id_{\La}$, i.e., with $\ad(x)_{\mid \La}=\id_{\La}$.  Choose a pre-Lie algebra
product $y_i\cdot y_j$ on $\Lg\Ll_n(\C)$, see \cite{BU4}. This is a post-Lie algebra structure on the pair $(\Lg\Ll_n(\C), \La)$.
Now extend this product to the pair $(\Lg\Ll_n(\C)\oplus \C, \Ln)$ by
\begin{align*}
x\cdot x & = 0,\\
y_i\cdot x & = 0 \\
x\cdot y_i & = -y_i=-\{x,y_i\}
\end{align*}
for all $1\le i \le n^2$. The axioms for a post-Lie algebra structure are satisfied. First, consider the identity
\[
u\cdot v-v\cdot u = [u,v]-\{u,v\}.
\]
We have to check it for all basis vectors $y_i,x$. It holds for all $u=y_i,v=y_j$, because $y_i\cdot y_j$ is a
pre-Lie algebra structure on $\Lg\Ll_n(\C)$ with $\{y_i,y_j\}=0$. So we only need to consider the case that
$u$ or $v$ is equal to $x$. By symmetry we may assume that $u=x$. Then for $v=x$ it trivially is true, and for $v=y_i$
we have
\[
x\cdot y_i-y_i\cdot x=-y_i=[x,y_i]-\{x,y_i\}.
\]
The third identity is equivalent to the fact that all $L(v)$ are derivations of $\Ln$ for all $v$. We have to check this
for $v=y_i$ and $v=x$. In these cases, the last row and last column of $L(v)$ are zero. Let us consider
the remaining matrices of size $n^2$. For $v=y_i$, this matrix is a derivation of $\La$,
since the product is a post-Lie algebra structure on $(\Lg\Ll_n(\C),\La)$. For $v=x$, this matrix is $-I$. In both cases
it follows that $L(v)\in \Der(\Ln)$. \\[0.2cm]
The second identity is equivalent to the fact that $L\colon x\ra L(x)$ is a Lie algebra representation of $\Lg$.
This is again obvious from the form of the operators $L(y_i)$ and $L(x)$. Indeed, $L(x)$ commutes with all operators
$L(y_i)$.
\end{proof}

We would like to prove a similar result as Theorem $\ref{5.1}$ with $\Ln$ nilpotent, non-abelian.
We start with the case of $\Lg=\mathfrak{gl}_2(\C)$.

\begin{prop}\label{5.3}
Let $(\Lg,\Ln)$ be a pair of Lie algebras with $\Lg=\mathfrak{gl}_2(\C)$ and $\Ln$ nilpotent, non-abelian. Then there is
no post-Lie algebra structure on $(\Lg,\Ln)$.  
\end{prop}

\begin{proof}
There are two non-abelian nilpotent Lie algebras of dimension $4$ up to isomorphism, namely $\Ln=\Ln_4(\C)$ and
$\Ln=\Ln_3(\C)\oplus \C$. Here $\Ln_4(\C)$ is the standard graded filiform Lie algebra, and $\Ln_3(\C)$ is the Heisenberg
Lie algebra. For  $\Ln=\Ln_4(\C)$, the Lie algebra $\Der(\Ln)$ is solvable, hence also the semidirect product $\Ln\rtimes \Der(\Ln)$.
Suppose that there exists a post-Lie algebra structure on $(\Lg,\Ln)$. Then by Proposition $2.11$ in \cite{BU41} there is an injective
homomorphism $\Lg \hookrightarrow \Ln\rtimes \Der(\Ln)$ given by $x\mapsto (x,L(x))$. It follows that $\Lg$ is solvable. However,
$\Lg=\mathfrak{gl}_2(\C)$ is not solvable, a contradiction. \\[0.2cm]
Secondly, let us assume that $\Ln=\Ln_3(\C)\oplus \C$. Suppose that $x\cdot y$ is a post-Lie algebra structure on $(\Lg,\Ln)$ and let
$\Ls=\mathfrak{sl}_2(\C)$. Then $\Ls$ acts on $\Ln$ by the restriction of the homomorphism $L\colon \Lg\ra \Der(\Ln)$ to $\Ls$.
We claim that $\Ls$ acts trivially on $Z(\Ln)$. Indeed, let $\{\Ln,\Ln\}=\langle z \rangle \subset Z(\Ln)=\langle z,t \rangle$. By Weyl's
Theorem, there is an $\Ls$-invariant complement $W$ of $\{\Ln,\Ln\}$ in $Z(\Ln)$. So we have $Z(\Ln)=\{\Ln,\Ln\}\oplus W$ as $\Ls$-modules
with $\dim (\{\Ln,\Ln\})=\dim(W)=1$. Since $\Ls$ is semisimple, it acts trivially on $W$ and $\{\Ln,\Ln\}$, so that $\Ls\cdot Z(\Ln)=0$.
Again by Weyl's theorem we have
\[
\Ln=Z(\Ln)\oplus U
\]
as $\Ls$-modules with a complement $U=\langle u,v\rangle$. This is a subalgebra of $\Ln$. Suppose that $\{U,U\}=0$. Then
$\{\Ln,\Ln\}=\{U\oplus Z(\Ln),U\oplus Z(\Ln)\}=0$, which is a contradiction. It follows that $\{u,v\}=w$ is nonzero
and $\Lh=\langle u,v,w\rangle$ is a subalgebra of $\Ln$ isomorphic to $\Ln_3(\C)$. By construction, $\Ls\cdot \Lh\subseteq \Lh$.
So by the first axiom for a post-Lie structure we have
\[
[x,y]=x\cdot y-y\cdot x+\{x,y\} \in \Lh
\]
for all $x,y\in \Lg$. It follows that $\Ls=[\Ls,\Ls]\subseteq \Lh$, so that both vector spaces for $\Ls$ and $\Lh$ coincide. Hence $x\cdot y$
induces a post-Lie algebra structure on $(\Ls,\Lh)$, where $\Ls$ is semisimple and $\Lh$ is 
nilpotent. This is a contradiction to Theorem $4.2$ of \cite{BU44}.
\end{proof}

We can generalize this result for pairs $(\mathfrak{gl}_n(\C), \Ln)$ where $\Ln$ is $2$-step nilpotent and non-abelian. 
For this we need the following lemma. 

\begin{lem}\label{5.4}
Let $x\cdot y$ be a PA-structure on a pair $(\Lg,\Ln)$, where $\Ln$ is $2$-step nilpotent with Lie bracket $\{x,y\}$.
Then
\[
x\circ y =\frac{1}{2}\{x,y\}+x\cdot y
\]
defines a pre-Lie algebra structure on $\Lg$.
\end{lem}  

\begin{proof}
This follows from Lemma $4.1$ and Proposition $4.2$ of \cite{BU41}. But the axioms \eqref{pre1} and \eqref{pre2}
for a pre-Lie algebra structure on $\Lg$ can also be easily verified directly.
\end{proof}

Defining the linear operators $\ell(x),L(x),\ad(x)$ by  $L(x)(y)=x\cdot y$, $\ell(x)(y)=x\circ y$ and $\ad(x)(y)=\{x,y\}$
we can rewrite the pre-Lie algebra product as
\[
\ell(x)=\frac{1}{2}\ad(x)+L(x).
\]

\begin{prop}
Let $(\Lg,\Ln)$ be a pair of Lie algebras where $\Lg=\mathfrak{gl}_n(\C)$ and $\Ln$ is $2$-nilpotent and non-abelian.
Then there is no post-Lie algebra structure on $(\Lg,\Ln)$.  
\end{prop}  

\begin{proof}
By Proposition $\ref{5.3}$ we may assume that $n\ge 3$. Assume that $x\cdot y=L(x)(y)$ is a
post-Lie algebra structure on $(\Lg,\Ln)$ and let $L\colon \mathfrak{gl}_n(\C)\ra \Der(\Ln)$
be the representation given by $x\mapsto L(x)$. Then the restriction to $\Ls=\mathfrak{sl}_n(\C)$ defines an $\Ls$-module $M_L$.
By Lemma $\ref{5.4}$ we obtain a $\Lg$-module, and then an $\Ls$-module $M_{\ell}$, by $\ell(x)=\frac{1}{2}\ad(x)+L(x)$.
This defines a pre-Lie algebra structure (or left-symmetric structure) on $\Lg$. \\
Now we can apply Theorem $4.5$ of \cite{BAU} for all $n\ge 3$. It implies that the $\Ls$-module $M_{\ell}$ is {\em special}
in the sense of \cite{BAU}, and equivalent to $M_n(\C)$. The module action here is given
by left multiplication of matrices, and $M_n(\C)$ is equivalent to $L(\om_1)^{\oplus n}$ and as well to $(L(\om_1)^*)^{\oplus n}$,
where $L(\om_1)$ denotes the $n$-dimensional natural $\Ls$-module and $L(\om_1)^*$ denotes its dual module.
We may assume that the only irreducible $\Ls$-submodule of  $M_{\ell}$ is of type $L(\om_1)$.
Since $\Ln$ is $2$-step nilpotent and $\{\Ln,\Ln\}$ is a characteristic subspace of $\Ln$, the formula
$\ell(x)=\frac{1}{2}\ad(x)+L(x)$ shows that $M_L$ and $M_{\ell}$ have the same irreducible $\Ls$-submodules.
So also $M_L$ has only irreducible submodules of type $L(\om_1)$. \\
On the other hand, $\{\Ln,\Ln\}$ is an $\Ls$-invariant submodule of $M_L$. By Weyl's theorem there exists an
$\Ls$-invariant complement $V=V_1\oplus \cdots \oplus V_k$, where $V_i$ is irreducible, with $\Ln=V\oplus \{\Ln,\Ln\}$.
So $V$ is a generating space for $\Ln$. Thus the $\Ls$-submodule $\{\Ln,\Ln\}$ is an $\Ls$-submodule of
\[
V\wedge V =(V_1\oplus \cdots \oplus V_k)\wedge (V_1\oplus \cdots \oplus V_k).
\]
By Lemma $3.4$ in \cite{BU74} we have
\[
V_1\wedge V_j=\begin{cases} L(\om_1)\otimes L(\om_1)\cong L(\om_2)\oplus L(2\om_1) & \text{ if } i\neq j,\\
 {\rm Sym}^2(L(\om_1))\cong L(2\om_1) & \text{ if } i=j.
 \end{cases}
\]
So $V\wedge V$ contains an irreducible $\Ls$-submodule $L(2\om_1)$, which is not of type $L(\om_1)$. It follows that either
$\{\Ln,\Ln\}=0$, which is a contradiction, or that the $\Ls$-submodule $\{\Ln,\Ln\}$ doesn't contain an irreducible $\Ls$-submodule
of type $L(\om_1)$ and hence cannot be an $\Ls$-submodule of $M_L$. This is also a contradiction.
\end{proof}

The method of this proof can also be applied to other pairs $(\Lg,\Ln)$, where $\Lg$ is reductive with $1$-dimensional center and
$\Ln$ is $2$-step nilpotent and non-abelian. However, for $\Ln$ being nilpotent of class $c\ge 3$ we do not know whether or not
there exists a post-Lie algebra structure on $(\Lg,\Ln)$. So we formulate an open question as follows:

\begin{qu}
Let $(\Lg,\Ln)$ be a pair of Lie algebras, where $\Lg$ is reductive non-abelian with $1$-dimensional 
center, and $\Ln$ is nilpotent non-abelian. Is it true that there are no post-Lie algebra structures 
on the pair $(\Lg,\Ln)$?
\end{qu}

\section*{Acknowledgments}
Dietrich Burde is supported by the Austrian Science Foun\-da\-tion FWF, grant I 3248 and grant 
P 33811. Mina Monadjem acknowledges support from the FWF grant P 33811. Karel Dekimpe is 
supported by a long term structural funding, the Methusalem grant of the Flemish Government.
We thank Wolfgang Moens for helpful discussions.

\end{document}